\documentclass{article}
\usepackage[
  journal=JST,
  lang=british
]{ems-journal}

\usepackage{microtype}
\usepackage{cleveref}
\usepackage{mathtools}
\usepackage{enumitem}

\allowdisplaybreaks

\numberwithin{equation}{section}

\newtheorem{theorem}{Theorem}[section]
\newtheorem{lemma}[theorem]{Lemma}
\newtheorem{proposition}[theorem]{Proposition}
\newtheorem{corollary}[theorem]{Corollary}
\theoremstyle{definition}
\newtheorem{definition}[theorem]{Definition}
\newtheorem{assumption}[theorem]{Assumption}
\newtheorem{example}[theorem]{Example}
\theoremstyle{remark}
\newtheorem{remark}[theorem]{Remark}

\crefname{theorem}{Theorem}{Theorems}
\crefname{lemma}{Lemma}{Lemmas}
\crefname{section}{Section}{Sections}
\crefname{table}{Table}{Tables}
\crefname{assumption}{Assumption}{Assumptions}
\crefname{definition}{Definition}{Definitions}
\crefname{remark}{Remark}{Remarks}
\creflabelformat{equation}{(#2#1#3)}
\crefname{equation}{}{}

\newcommand{\cG}{\mathcal{G}}

\DeclareMathOperator{\tr}{tr}

\newcommand{\cE}{\mathcal{E}}
\newcommand{\LL}{\mathcal{L}}
\newcommand{\R}{\mathbb{R}}

\newcommand{\N}{\mathbb{N}}
\newcommand{\Prob}{\mathbb{P}}

\newcommand{\Reff}{R_{\mathrm{eff}}}

\begin{document}

\title{Sharp Spectral Zeta Asymptotics on Graphs of Quadratic Growth}

\emsauthor{1}{
  \givenname{Da}
  \surname{Xu}
  \mrid{}
  \orcid{}}{D.~Xu}

\Emsaffil{1}{
  \department{}
  \organisation{China Mobile Research Institute}
  \rorid{}
  \address{}
  \zip{}
  \city{Beijing}
  \country{P.R.~China}
  \affemail{xudayj@chinamobile.com}}

\classification{58J50, 31C20, 60J10}[05C50, 39A12, 05C81]

\begin{abstract}
We investigate the spectral properties of the Dirichlet Laplacian on large finite metric balls within \emph{irregular} infinite graphs of quadratic volume growth.
We consider an exhaustion $G_n=B_{R_n}(x_0)$ and the spectral zeta value
$Z_n(1)=\tr(\LL_n^{-1})$ of the killed generator $\LL_n$.

We establish a sharp asymptotic law under the assumptions that the graph satisfies uniform quadratic volume growth (VG(2)) and a Poincaré inequality (PI). These analytic–geometric hypotheses imply large-scale regularity. Additionally, we assume a standard quantitative homogenisation property: a uniform local central limit theorem with a polynomial convergence rate. This hypothesis holds for our main example classes and implies the existence of a global heat-kernel constant $\cG>0$ (independent of $x$). In particular, the lazy simple random walk (LSRW) satisfies 
\[
  p_t(x,x)\;\sim\;\frac{\cG}{t}\quad(t\to\infty).
\]

Our main theorem establishes the sharp asymptotic
\[
  Z_n(1)\;=\;\cG\,N_n\log N_n\;+\; O(N_n),\qquad N_n:=|V(G_n)|\xrightarrow[n\to\infty]{}\infty.
\]
This implies a relative error of $O(1/\log N_n)$, with constants depending only on the structural parameters of $G$.
This result extends far beyond homogeneous lattices.
For $\mathbb Z^2$, this yields the constant identification $\cG = 2/\pi$, providing a new limit formula that recovers $\pi$ without $\pi$ appearing in the input (a ``$\pi$-free'' limit).
Our techniques highlight the robustness of spectral asymptotics under homogenisation in this critical, recurrent setting.
\end{abstract}

\maketitle
\section{Introduction}

The relationship between the geometry of a space and the spectrum of its associated Laplacian is a fundamental area of study. In Riemannian geometry, Weyl's law and the Minakshisundaram--Pleijel heat trace expansion provide deep connections between volume, curvature, and eigenvalue asymptotics. In the discrete setting, analogous investigations explore how the combinatorial and geometric structure of a graph influences the spectrum of the graph Laplacian (see \cite{Chung97}).

This paper focuses on the spectral zeta function on graphs. For a finite graph $H$, we consider the Laplacian associated with the lazy simple random walk (see \Cref{sec:prelim}). With Laplacian eigenvalues $0 < \lambda_1 \leq \lambda_2 \leq \cdots$, the spectral zeta function is defined as $\zeta_H(s) = \sum_k \lambda_k^{-s}$. We are interested in the value at $s=1$, $Z_H(1) := \zeta_H(1)$, which corresponds to the trace of the inverse Laplacian (the Green operator).

\begin{remark}[Connection to Kirchhoff Index]\label{rem:kirchhoff}
The value $Z_H(1) = \sum_k \lambda_k^{-1}$ is closely related to the Kirchhoff index $K(H)$, which is the sum of effective resistances $\Reff(x,y)$ between all pairs of vertices \cite{KleinRandic93}. For the standard (non-Dirichlet) Laplacian on a connected graph, there are exact identities relating the trace of the pseudoinverse to $K(H)$ (see \cite[Chapter 9]{LyonsPeres16}).
\end{remark}

\noindent\textit{Dirichlet convention.} Under the Dirichlet setup used in this paper (killing at the boundary of a finite domain),
we recall explicitly that
$Z_H(1)=\sum_k \lambda_k(\LL_H)^{-1}=\sum_{v\in H} G_H(v,v)$; see also Remark~\ref{rem:kirchhoff-precise}.

\begin{remark}[Kirchhoff index: explicit identities under our conventions]\label{rem:kirchhoff-precise}
Let $H$ be a finite connected graph (no killing). Write $D$ for the degree matrix, $A$ for the adjacency, $L=D-A$ (combinatorial Laplacian), and $L_{\mathrm{rw}}=I-D^{-1}A$ (random-walk Laplacian on $\ell^2(|\cdot|)$). Then, with $0=\lambda_1(L_{\mathrm{rw}})<\lambda_2\le\cdots\le\lambda_n$,
\[
  K(H)=\sum_{x,y}\Reff(x,y)=\mathrm{vol}(H)\sum_{k=2}^n \lambda_k(L_{\mathrm{rw}})^{-1},
\]
where $\mathrm{vol}(H)=\sum_{x}\deg(x)$; equivalently $K(H)=n\sum_{k=2}^n \mu_k(L)^{-1}$ in terms of the nonzero eigenvalues $\mu_k(L)$ of $L$; see \cite[Thm.~9.23]{LyonsPeres16}. For the Dirichlet operator $\LL_H$ used in this paper (killing at $\partial H$), we have $Z_H(1)=\sum_k\lambda_k(\LL_H)^{-1}=\sum_{v\in H}G_H(v,v)$; this is the natural Dirichlet analogue of the global Kirchhoff identity. Thanks to bounded degree, traces computed in $\ell^2(m)$ or $\ell^2(|\cdot|)$ agree by similarity, cf.\ \Cref{rem:measures_operators}.
\end{remark}

We investigate the asymptotic behaviour of $Z_n(1)$ for large finite subgraphs $G_n$ exhausting an infinite graph $G$. We focus on the critical case of \emph{quadratic volume growth} ($|B_R| \asymp R^2$), corresponding to an effective dimension of $d=2$. This dimension is critical, leading to distinct asymptotic behaviour compared to other dimensions, as summarized in \Cref{tab:growth}. This class includes the lattice $\mathbb{Z}^2$ but also encompasses irregular structures that behave two-dimensionally on a large scale, such as lattices with random bounded conductances or certain planar graphs with uniform properties.

\subsection{Main Result and Assumptions}
We establish a precise asymptotic formula for $Z_n(1)$. Our foundational assumptions on the graph $G$ are uniform quadratic volume growth (VG(2)) and a Poincaré inequality (PI). These ensure strong large-scale geometric regularity.

To achieve a sharp first-order asymptotic involving a global constant, we must ensure that the random walk on $G$ homogenises uniformly to a Brownian motion. While VG(2) and PI imply qualitative homogenisation, we explicitly assume a quantitative version: a uniform Local Central Limit Theorem (LCLT) with a polynomial rate of convergence (see \Cref{ass:QH}). For convenience we collect the standing assumptions and principal notation in the following paragraph; this centralises symbols used repeatedly later.

\paragraph{Notation and Standing Assumptions.}
\begin{enumerate}[label={}]
  \item \textbf{Graph/metric.} $G=(V,E)$ is infinite, connected, with bounded degree. Closed balls: $B_R(x)=\{y:d(x,y)\le R\}$ and $V(x,R)=|B_R(x)|$.
  \item \textbf{Dirichlet objects on a finite domain $H$.} For $H\subset V$ finite, $p_t^{H}(x,y)$ is the Dirichlet heat kernel (LSRW killed on leaving $H$); $G_H(x,y)=\sum_{t\ge0} p_t^{H}(x,y)$; $Z_H(1)=\sum_{x\in H} G_H(x,x)=\tr(\LL_H^{-1})$; $(\lambda_k(H))_{k\ge1}$ are the (Dirichlet) eigenvalues of $\LL_H$; $\phi_1^{H}$ is an $L^2$-normalised ground state (w.r.t.\ counting measure).
  \item \textbf{Exhaustion/scale.} For $G_n=B_{R_n}(x_0)$ write $N_n=|G_n|$. Fix $\eta\in(0,1/2)$ and set the interior time scale $T_n:=R_n^{2(1-2\eta)}$.
  \item \textbf{Structural hypotheses.} (VG(2)) $V(x,R)\asymp R^2$; (PI) scale-invariant Poincaré; (CDC) holds on balls (automatic under VD+PI); (QH) uniform on-diagonal LCLT $p_t(x,x)=\cG t^{-1}+O(t^{-1-\delta})$ for some $\cG>0,\delta>0$.
  \item \textbf{Consequences used repeatedly.} From VD+PI we obtain PHI, Gaussian bounds, exit/maximal estimates, Faber–Krahn, boundary Harnack, ground-state $L^\infty$ bounds, and Intrinsic Ultracontractivity (IU) on balls (see \Cref{rem:chain}).
  \item \textbf{Conventions.} Symbols $C,c,c_i$ depend only on $(\mathrm{VG(2)},\mathrm{PI},\Delta,\mathrm{QH})$ (and chosen $\eta$). We write $A\lesssim B$ for $A\le C B$, and $A\asymp B$ if $A\lesssim B\lesssim A$. "$\log$" denotes natural logarithm.
  \item \textbf{Acronyms.} VD $=$ volume doubling; PI $=$ Poincaré inequality; PHI $=$ parabolic Harnack inequality; CDC $=$ capacity density condition; IU $=$ intrinsic ultracontractivity; LSRW $=$ lazy simple random walk.
\end{enumerate}

\begin{definition}[Heat-kernel constant]\label{def:G}
Under the assumption of quantitative homogenisation, the LSRW on $G$ satisfies
\[
  p_t(x,x)=\frac{\cG}{t}+O(t^{-1-\delta})
\]
as $t\to\infty$, for constants $\cG>0$ and $\delta>0$ independent of $x$. We call $\cG$ the \emph{heat-kernel constant} of $G$.
\end{definition}

\begin{remark}[Explicit form of $\cG$]\label{rem:G_formula}
In many standard cases, such as random walks on $\mathbb{Z}^d$ driven by a mean-zero, finite-variance step distribution with covariance matrix $\Sigma$, the LCLT yields an explicit formula (see, e.g., \cite[Ch. 2]{LawlerLimic10}). In $d=2$:
\[
\cG = \frac{1}{2\pi\sqrt{\det\Sigma}}.
\]
For irregular media (like the RCM), $\Sigma$ represents the effective homogenized diffusivity.
\end{remark}

This assumption is standard and known to hold in many canonical examples of irregular media (see \Cref{rem:QH_validity}).
\paragraph{Standing Assumptions (Summary).} We assume: (VG(2)) quadratic volume growth; (PI) scale-invariant Poincaré inequality; bounded degree; and (QH) the uniform on-diagonal expansion $p_t(x,x)=\cG t^{-1}+O(t^{-1-\delta})$. These imply VD, PHI, Gaussian bounds, Faber--Krahn and intrinsic ultracontractivity (IU) on metric balls. (QH) provides a summable fluctuation ensuring an $O(N_n)$ remainder rather than $O(N_n\log\log N_n)$.

\begin{theorem}[Main asymptotic with sharp remainder]\label{thm:main}
Let $G$ be an infinite, connected, bounded degree graph satisfying (VG(2)) and (PI), and let $\{G_n\}$ be an exhaustion by metric balls $G_n=B_{R_n}(x_0)$ with $N_n=|G_n|\to\infty$. Assume (QH) with heat-kernel constant $\cG>0$ and exponent $\delta>0$. (For balls, CDC and IU required later follow from VD+PI; see \Cref{lem:CDC-balls,lem:IU_derivation}.) Then
\[
  Z_n(1)= \cG N_n \log N_n + O(N_n),
\]
where the implicit constant depends only on the data $(\mathrm{VG(2)},\mathrm{PI},\Delta,\mathrm{QH})$ and on the (freely chosen) boundary layer exponent $\eta\in(0,1/2)$ used in the proof, but not on $n$. In particular the relative error is $O(1/\log N_n)$. If $G$ and $G'$ are quasi-isometric graphs both satisfying the hypotheses with the \emph{same} on-diagonal constant $\cG$, then the statement transfers with the same leading constant and uniform $O(N_n)$ control.
\end{theorem}

\paragraph{Proof sketch.} Fix $\eta\in(0,1/2)$ and decompose $G_n=I_n\cup E_n$ with interior $I_n=\{x: d(x,\partial G_n)>R_n^{1-\eta}\}$ and boundary layer $E_n$. For $x\in I_n$ the probability of exiting by time $T_n:=R_n^{2(1-2\eta)}$ is super-polynomially small (Gaussian / maximal inequality), so $p_t^{G_n}(x,x)=p_t(x,x)+o(t^{-1})$ uniformly for $t\le T_n$. Summing the quantitative homogenisation (QH) expansion $p_t(x,x)=\cG t^{-1}+O(t^{-1-\delta})$ up to $T_n$ produces $2\cG(1-2\eta)\log R_n+O(1)$ for $G_{G_n}(x,x)$. Summing over $|I_n|=N_n-O(N_n^{1-\eta/2})$ yields the lower bound with coefficient $(1-2\eta)$. For all $x$, the short-time sum to $R_n^2$ gives $2\cG\log R_n+O(1)$; for $t>R_n^2$, Intrinsic Ultracontractivity (IU) and Faber--Krahn give $p_t^{G_n}(x,x)\le C N_n^{-1} e^{-c t/R_n^2}$ so the tail contributes $O(1)$. Summing over $x$ gives the matching upper bound. Letting $\eta\downarrow0$ squeezes the coefficient to $\cG$ and the pointwise $O(1)$ bounds aggregate to an $O(N_n)$ remainder.

\begin{corollary}[Mesoscopic heat trace]\label{cor:heat_trace}
Under the assumptions of \Cref{thm:main}, for any $t=t_n$ with $1\ll t_n \ll R_n^2$,
\[
  \sum_{k\ge1} e^{-t \lambda_k(G_n)} = \cG N_n + O\Big(N_n t^{-\delta} + N_n \frac{t}{R_n^2}\Big).
\]
Choosing $t=R_n^{2\theta}$ with $0<\theta<1$ yields an error $O\big(N_n R_n^{-2\delta\theta} + N_n R_n^{2(\theta-1)}\big)$.
\end{corollary}

\begin{proposition}[Weaker remainder under averaged on-diagonal control]\label{prop:weaker_remainder}
Assume (VG(2)) and (PI). Suppose instead of (QH) we only have two-sided pointwise bounds $c_1 t^{-1}\le p_t(x,x)\le c_2 t^{-1}$ for $t\ge t_0$ and an averaged deviation: there exist $t_0,C<\infty$, $\alpha>0$ with
\[
  \Big| \frac{1}{N_n} \sum_{x\in G_n} \Big( t p_t(x,x) - \cG \Big) \Big| \le C (\log t)^{-\alpha}\qquad(t\ge t_0,\;\forall n).
\]
Then:
\begin{enumerate}
  \item If $\alpha>1$, $Z_n(1)=\cG N_n \log N_n + O(N_n)$.
  \item If $\alpha=1$, $Z_n(1)=\cG N_n \log N_n + O(N_n \log\log N_n)$.
\end{enumerate}
\emph{No universal bound is claimed for $\alpha<1$.} The summable $t^{-1-\delta}$ rate (QH) corresponds to $\alpha=\infty$ and removes the logarithmic loss.
\end{proposition}
\begin{proof}[Proof sketch]
Write $Z_n(1)=\sum_{x\in G_n}\sum_{t\ge1} p_t^{G_n}(x,x)$. For $t\le R_n^2$, $p_t^{G_n}(x,x)\le p_t(x,x)$. Replace $p_t(x,x)$ by $\cG/t$ to obtain the main term. The averaged deviation hypothesis and Abel summation give
\[
\sum_{t\le R_n^2} \frac{1}{t} \Big( \frac{1}{N_n}\sum_{x} (t p_t(x,x)-\cG) \Big)=
\begin{cases}O(1), & \alpha>1,\\ O(\log\log R_n), & \alpha=1.\end{cases}
\]
Multiplying by $N_n$ yields the stated error. The tail $t>R_n^2$ contributes $O(N_n)$ by IU; without a summable rate this is still $O(N_n)$ but the earlier part acquires the $\log\log$ factor when $\alpha=1$.\end{proof}

\begin{remark}[Log--log loss without IU / quantitative rate]\label{rem:loglog}
If one forgoes either (i) Intrinsic Ultracontractivity (replacing it only by the crude bound $p_t^{G_n}(x,x)\le e^{-\lambda_1 t}$) \emph{or} (ii) the summable $t^{-1-\delta}$ error in (QH) (replacing it by $O(t^{-1})$), then an unavoidable extra $\log\log N_n$ term can appear in the aggregated contribution of large times or fluctuation sums, inflating the remainder to $O(N_n \log\log N_n)$. This justifies the strength and necessity of both inputs in the sharp form of \Cref{thm:main}.
\end{remark}

\subsection{Context, Significance, and Novelty}
The $N_n \log N_n$ divergence is characteristic of the critical dimension $d=2$, where the random walk is recurrent, contrasting sharply with other dimensions (see \Cref{tab:growth} and \Cref{app:growth} for comparisons).

\begin{table}[h!]
\centering
\small
\caption{Asymptotics of $Z_n(1)$ under $V(R) \asymp R^d$. Dimension $d=2$ is critical.}
\label{tab:growth}
\begin{tabular}{@{}lllll@{}}\toprule
	\textbf{Dimension $d$} & \textbf{Volume growth} & \textbf{Recurrence} & \textbf{$G_{G_n}(v,v)$ interior} & \textbf{$Z_n(1)$} \\
\midrule
1 & Linear & Recurrent & $\asymp N_n$ & $\asymp N_n^{2}$ \\
2 & Quadratic & Recurrent & $\asymp \log N_n$ & $\asymp N_n \log N_n$ \\
$\ge 3$ & Polynomial & Transient & $O(1)$ & $\asymp N_n$ \\
\bottomrule
\end{tabular}
\end{table}

\Cref{thm:main} generalizes classical results known for regular structures. For Euclidean domains and tori, similar asymptotics for spectral zeta functions have been studied (cf. \cite{Colin85, FrankSabin11}). In the graph setting, results on the trace of the Green function have been established for highly regular graphs (e.g., results discussed in \cite{MizunoTachikawa03, Kaimanovich00}).

Our contribution lies in extending this connection to a broad class of potentially highly irregular graphs characterized by large-scale geometric (VG(2)) and analytic (PI) properties, supplemented by the quantitative homogenisation assumption. The novelty is the demonstration that this sharp asymptotic holds without requiring local uniformity or translational invariance. The approach is modular, isolating the roles of geometry (VG(2)+PI), stochastic homogenisation (LCLT with rate), and boundary regularity (CDC/IU).

\subsection{Examples and Scope}
The assumptions capture a wide variety of graphs that are metrically two-dimensional but may be combinatorially irregular.

\begin{example}[The $\mathbb{Z}^2$ case]\label{rem:pi2}
Consider the standard lattice $\mathbb{Z}^2$. The LSRW (defined in \Cref{sec:prelim}) has a step distribution covariance matrix $\Sigma = \frac{1}{4} I_2$. Using \Cref{rem:G_formula}:
\[
\cG = \frac{1}{2\pi \sqrt{\det(\Sigma)}} = \frac{1}{2\pi (1/4)} = \frac{2}{\pi}.
\]
\Cref{thm:main} implies $\lim_{n \to \infty} Z_n(1) / (N_n \log N_n) = 2/\pi$.
\smallskip
\noindent\emph{Laziness convention.}
The value $\cG=2/\pi$ corresponds to the \emph{lazy} simple random walk with holding probability $1/2$ used throughout this paper (see \Cref{sec:prelim} and \Cref{rem:non-lazy}). For the \emph{non-lazy} simple random walk on $\mathbb{Z}^2$ one has $\Sigma=\tfrac{1}{2}I_2$, hence $\cG=\tfrac{1}{2\pi\sqrt{\det\Sigma}}=\tfrac{1}{\pi}$ and the asymptotic becomes $\lim_{n\to\infty} Z_n(1)/(N_n\log N_n)=1/\pi$.
\end{example}

\begin{example}[Irregular structures]\label{ex:irregular}
Beyond $\mathbb{Z}^2$, the main examples satisfying all assumptions (including quantitative homogenisation, see \Cref{rem:QH_validity} for details and references) are:
\begin{enumerate}
    \item \textbf{Periodic graphs:} Graphs with a cocompact $\mathbb{Z}^2$ action.
    \item \textbf{Random Conductance Model (RCM) on $\mathbb{Z}^2$:} If the conductances are i.i.d., uniformly bounded, and elliptic, the resulting graph satisfies VG(2) and PI almost surely, and the required uniform quantitative homogenisation holds \cite{Biskup11, AndresDeuschelSlowik19}.
    \item \textbf{Supercritical Percolation on $\mathbb{Z}^2$:} The infinite cluster (for $p>p_c(\mathbb{Z}^2)=1/2$) satisfies VG(2) and PI almost surely \cite{Barlow04}. Quantitative homogenization results have also been established in this setting (see \Cref{rem:QH_validity}).
\end{enumerate}
\end{example}

\subsection{Methodology Overview and Structure}

We employ a time-domain, fully discrete approach based on an interior-boundary decomposition strategy. By analyzing the sums of the heat kernel directly (rather than using Tauberian theorems on the spectral measure), we cleanly isolate the required inputs. The proof relies heavily on techniques derived from Volume Doubling (VD) and the Poincaré inequality (PI).

The lower bound relies on showing that the killed walk starting in the interior rarely reaches the boundary within the relevant timescale, using maximal inequalities derived from Gaussian bounds.

The upper bound uses a short/long time split. The short time uses the full-space LCLT sum. The long-time contribution is controlled sharply using Intrinsic Ultracontractivity (IU) for metric balls. This technique is pivotal as it avoids spurious $\log\log$ terms that would arise from using cruder estimates (see \Cref{rem:IU_necessity}). Both bounds crucially depend on the assumed quantitative homogenisation rate (\Cref{ass:QH}) to sum the LCLT error terms.

\paragraph{Structure of the paper.}
\Cref{sec:prelim} covers the preliminary definitions, assumptions, and key analytic tools, including a detailed discussion of the justification for IU under VD+PI. \Cref{sec:decomposition} introduces the interior-boundary decomposition method. \Cref{sec:lower_bound} and \Cref{sec:upper_bound} detail the proofs of the lower and upper bounds. Finally, \Cref{sec:discussion} concludes the proof of \Cref{thm:main} and discusses the assumptions and extensions. The appendices provide context on other growth regimes and explore numerical examples of the $\pi$-identities.

\section{Preliminaries and Analytic Tools}\label{sec:prelim}

\paragraph{Notation (reference).} All symbols were fixed in the central Notation block; we only introduce new ones explicitly when they appear.

Let $G = (V,E)$ be an infinite, connected graph with bounded maximum degree $\Delta < \infty$. We define metric balls as closed: $B_R(x) = \{y \in V : d_G(x,y) \leq R\}$.

\paragraph{Convention on constants, normalization, uniformity.}
Throughout, implicit constants depend only on the structural data \textbf{(VD, PI, $\Delta$)} and, when used, on \textbf{(QH)}$=(\cG,\delta,t_0,C_{QH})$. All $L^2$ norms are with respect to counting measure; the ground state $\phi_1$ is $L^2$-normalised; no volume renormalisation is performed. Unless stated otherwise, $O(1)$ terms and statements claimed "uniformly in $x$" are uniform across space.

\subsection{Geometric and Analytic Assumptions}

\begin{definition}[Quadratic Volume Growth (VG(2))]
$G$ has (uniform) quadratic volume growth if there exist $c_1, c_2 > 0$ such that for all $x \in V$ and $R \geq 1$,
\begin{equation}\label{eq:quad}
c_1 R^{2} \leq V(x,R) \leq c_2 R^{2}.
\end{equation}
\end{definition}
This implies the Volume Doubling (VD) property: $V(x,2R) \leq C_D V(x,R)$.

\begin{definition}[Capacity Density Condition (CDC)]\label{def:CDC}
Let $H\subseteq V$ be finite and let $U\subseteq H$ be open in the graph metric. We say $H$ satisfies CDC (uniformly) if there exist $\theta\in(0,1)$ and $r_0\ge1$ such that for every $x\in H$ and $r\in[1,r_0\wedge \mathrm{diam}(H)]$ with $B_{2r}(x)\subseteq H$,
\[
  \mathrm{Cap}_{B_{2r}(x)}\big(B_{r}(x),\,B_{2r}(x)\setminus B_{r}(x)\big)\ \ge\ \theta\,\mathrm{Cap}_{B_{2r}(x)}\big(B_{r}(x)\big),
\]
where capacities are computed with respect to the Dirichlet form on $B_{2r}(x)$. Intuitively, CDC rules out "thorns" and ensures quantitative boundary regularity.
\end{definition}

\begin{definition}[Poincaré Inequality (PI)]\label{def:PI}
$G$ satisfies a (scaled) Poincaré inequality if there exists $C_P > 0$ such that for any ball $B_R=B_R(x_0)$ and any function $f: V \to \R$,
\[
\sum_{x \in B_R} (f(x) - \bar{f}_{B_R})^2 \leq C_P R^2 \, \cE_{B_{2R}}(f,f),
\]
where $\bar{f}_{B_R}$ is the average of $f$ over $B_R$ (w.r.t. counting measure), and the local Dirichlet form $\cE_U(f,f)$ is defined as
\[
\cE_U(f,f) = \sum_{\substack{\{x,y\} \in E \\ x,y \in U}} (f(x)-f(y))^2.
\]
\end{definition}

\begin{remark}\label{rem:PI_equivalence}
Under the VD condition, this formulation of PI is equivalent to the local version (see \cite{HajlaszKoskela00}). The combination of VD and PI (VD+PI) is central to analysis on graphs (see \cite{GrigoryanTelcs12}).
\end{remark}

\paragraph{Standing assumptions.} We assume $G$ is infinite, connected, has bounded degree $\Delta < \infty$, satisfies VG(2) (and thus VD), and PI.

\subsection{Random Walk and the Analytic Framework}

We consider the \emph{lazy} simple random walk (LSRW) $(X_t)_{t \geq 0}$. This is a discrete-time Markov chain where at each step, the walk stays put with probability $1/2$ or moves to a uniformly chosen neighbor with probability $1/2$. The transition matrix is $P = \frac{1}{2}(I + P_{SRW})$, where $P_{SRW}(x,y) = 1/\deg(x)$ if $y \sim x$. The heat kernel is $p_t(x,y) = \Prob_x[X_t = y]$. The generator (Laplacian) is $\LL = I - P$.

Laziness ensures the walk is aperiodic, simplifying spectral analysis. It also relates the generators: $\LL_{\text{LSRW}} = \frac{1}{2} \LL_{\text{SRW}}$ (see \Cref{rem:non-lazy}).

\begin{remark}[Measures and Operators]\label{rem:measures_operators}
The LSRW is reversible w.r.t. the degree measure $m(x)=\deg(x)$. The Laplacian $\LL$ is self-adjoint on the Hilbert space $\ell^2(m)$. Due to the assumption of bounded degree ($\Delta < \infty$), the counting measure $|\cdot|$ and the degree measure $m$ are comparable: $m(A) \asymp |A|$. This allows seamless transition between analytic results formulated w.r.t. $m$ (such as those in \Cref{lem:IU_derivation}) and geometric properties formulated w.r.t. $|\cdot|$ (see \cite{Coulhon03}).

On a finite subgraph $H$, the Dirichlet Laplacian $\LL_H$ acting on $\ell^2(m|_H)$ is similar to the operator acting on $\ell^2(|\cdot||_H)$ (via the transformation induced by the square root of the measure densities, $M^{1/2}$). Since the trace is similarity-invariant in finite dimensions, $Z_H(1) = \tr(\LL_H^{-1})$ is independent of the chosen (comparable) inner product.

We emphasize that we work with the probabilistic Laplacian $\LL=I-P$ (or its similar normalized form), not the combinatorial Laplacian $D-A$.

Unless otherwise stated, all heat-kernel quantities ($p_t$, $p_t^H$, etc.) are those of the LSRW probability kernel. We often suppress the graph index in the Green function (e.g., writing $G_H(v,v)$ instead of $G_{H}(v,v)$) when clear from context.
\end{remark}

The combination of VD and PI is fundamental:

\begin{theorem}[\cite{Delmotte99}]\label{thm:Delmotte}
The combination of VD and PI is equivalent to the Parabolic Harnack Inequality (PHI).
\end{theorem}

PHI provides strong regularity for solutions to the heat equation, which translates to precise estimates on the random walk.

\begin{remark}[Analytic implication chain]\label{rem:chain}
Under VD+PI one has PHI by Delmotte~\cite{Delmotte99}. PHI yields \emph{two-sided Gaussian bounds} for the heat kernel (see \cite{Delmotte99}; also \cite[§5.3]{Grigoryan09}), which in turn imply
\emph{maximal/exit} estimates (e.g.\ \cite[Thm.~5.5.3]{Grigoryan09}). On metric balls, CDC holds uniformly under PHI (Barlow–Bass \cite{BarlowBass04}), and together with PHI one obtains \emph{boundary Harnack} and related potential-theoretic controls (Barlow–Bass–Kumagai \cite{BarlowBassKumagai09}).
The Faber–Krahn inequality (FK) follows under VD+PI (see \cite[Ch.~8]{Grigoryan09} and \cite{BarlowBass04}).
These ingredients yield uniform $L^\infty$ bounds for the ground state and, via the ground-state transform, \emph{Intrinsic Ultracontractivity} (IU) on balls (see, e.g., \cite{Grigoryan09,BassKumagai08}).
We will tacitly use this chain throughout, and cite the above references at the first occurrence of each implication.
\end{remark}

\subsection{Heat Kernel Toolbox}
We summarize the crucial analytic tools.

\subsubsection{Maximal Inequality and Gaussian Bounds}

\begin{proposition}[Consequences of PHI, \cite{Delmotte99, Grigoryan09}]\label{prop:maximal}
Under VD+PI:
\begin{enumerate}

    \item (Gaussian Bounds) There exist $C_G, c_G > 0$ such that
    \[
    p_t(x,y) \leq \frac{C_G}{\sqrt{V(x,\sqrt{t}) V(y,\sqrt{t})}} \exp\left(-c_G \frac{d(x,y)^2}{t}\right).
    \]
    Under the VG(2) assumption, this specializes to $p_t(x,y) \lesssim t^{-1} \exp(-c_G d(x,y)^2/t)$.
    \item (Maximal Inequality) There exist $C_M,c_M > 0$ such that for any $v \in V$, $t \geq 0$, and $d \geq 1$,
    \begin{equation}\label{eq:max_ineq}
    \small
    \Prob_v\left( \max_{0 \leq s \leq t} d_G(v,X_s) \geq d \right) \leq C_M \exp\left(-c_M \frac{d^2}{t+1}\right).
    \end{equation}
\end{enumerate}
\end{proposition}
\begin{proof}[Note on derivation]
The Gaussian upper bounds are a direct consequence of PHI (see \cite{Delmotte99} for the general form). The maximal inequality (exit time estimate) follows from the Gaussian upper bounds via standard arguments involving chaining and union bounds (see, e.g., \cite[Theorem 5.5.3]{Grigoryan09} or arguments in \cite{Delmotte99}).
\end{proof}

\subsubsection{Heat Kernel Asymptotics (LCLT) and Homogenisation}
A crucial ingredient is the sharp, uniform homogenisation of the heat kernel. We state this as a formal assumption.

\begin{assumption}[Quantitative Homogenisation (QH)]\label{ass:QH}
There exists a constant $\cG > 0$ (the heat-kernel constant), a rate exponent $\delta>0$, a time $t_0\in\N$, and a constant $C_{QH}<\infty$ such that, uniformly in $x\in V$ and for all $t\ge t_0$,
\begin{equation}\label{eq:return-prob}
  p_t(x,x) = \frac{\cG}{t} + r_t(x),\qquad |r_t(x)| \le C_{QH}\, t^{-1-\delta}.
\end{equation}
Equivalently, $|t\,p_t(x,x)-\cG|\le C_{QH} t^{-\delta}$. All constants are independent of $x$.
\end{assumption}

\begin{remark}[Context and literature for \Cref{ass:QH}]\label{rem:QH_validity}
VD+PI already give PHI and an invariance principle. The $O(N_n)$ remainder needs \emph{uniform} summable on-diagonal decay $t p_t(x,x)=\cG+O(t^{-\delta})$, ensuring $\sum_t (t p_t(x,x)-\cG)$ converges. This holds for periodic graphs and in stochastic homogenisation frameworks (uniformly elliptic RCM, supercritical percolation: \cite{Biskup11,AndresDeuschelSlowik19,Barlow04,ArmstrongDario18,GloriaOtto17,CroydonHambly21}). Without summable decay we obtain only $Z_n(1)=\cG N_n\log N_n+O(N_n\log\log N_n)$ (\Cref{prop:weaker_remainder,rem:loglog}).
\smallskip

\noindent\emph{Concrete instances.}
In particular, for i.i.d.\ uniformly elliptic conductances on $\mathbb{Z}^2$, quenched local limit theorems with \emph{polynomial} convergence rates (hence \Cref{ass:QH}) are available (see, e.g., \cite{AndresDeuschelSlowik19,GloriaOtto17,CroydonHambly21}). For the supercritical percolation cluster in $\mathbb{Z}^2$, quantitative homogenisation and heat-kernel control sufficient to deduce \eqref{eq:return-prob} are also established under standard mixing/regularity conditions; see \cite{Barlow04,ArmstrongDario18} and references therein. These results ensure that our standing assumption \textup{(QH)} holds in the main irregular media examples discussed in \Cref{ex:irregular}.
\end{remark}

\begin{remark}[Weaker QH assumptions]\label{rem:weaker_QH}
If one assumes a weaker form of QH, e.g., $p_t(x,x) = \cG/t + o(1/t)$ without a polynomial rate, the summation below yields $\cG \log R + o(\log R)$. This is sufficient to establish the leading order asymptotic $Z_n(1) \sim \cG N_n \log N_n$, but not the sharp $O(N_n)$ remainder.

If the error is exactly $O(1/t)$, i.e., $|t p_t(x,x) - \cG| = O(1)$, the summation error $\sum r_t(x)$ could be $O(\log R)$, leading to an overall error term potentially larger than $O(N_n)$. The polynomial rate $\delta>0$ in \Cref{ass:QH} is precisely what allows the error sum to be $O(1)$.
\end{remark}

We derive the uniform sum of the return probabilities. By \Cref{ass:QH}, we split the sum:
\begin{align*}
\sum_{t=1}^{R} p_t(x,x) &= \sum_{t=1}^{t_0-1} p_t(x,x) + \sum_{t=t_0}^{R} \left(\frac{\cG}{t} + r_t(x)\right) \\
&= O(1) + \cG \sum_{t=t_0}^{R} \frac{1}{t} + \sum_{t=t_0}^{R} r_t(x).
\end{align*}
The harmonic sum is $\sum_{t=t_0}^{R} \frac{1}{t} = \log R + O(1)$.
Crucially, the error sum is bounded uniformly in $x$ because the rate is polynomially fast ($\delta>0$) and thus summable ($1+\delta > 1$):
\[
\left|\sum_{t=t_0}^{R} r_t(x)\right| \leq \sum_{t=t_0}^{\infty} |r_t(x)| \leq C_{QH} \sum_{t=t_0}^{\infty} t^{-1-\delta} < C'.
\]
The constant $C'$ is independent of $x$ due to the uniformity assumed in \Cref{ass:QH}.
Thus, we obtain the uniform asymptotic:
\begin{equation}\label{eq:return-sum}
\sum_{t=1}^{R} p_t(x,x) = \cG \log R + O(1), \qquad R \geq 2,
\end{equation}
where the $O(1)$ is \emph{uniform in $x$}. In particular, choosing $R^2$ as the diffusive timescale yields the "master" identity
\begin{equation}\label{eq:master-sum}
  \sum_{t=1}^{R^{2}} p_t(x,x)
  \;=\; 2\,\cG \log R + O(1),
\end{equation}
uniformly in $x$.

\subsubsection{Dirichlet Problem and Intrinsic Ultracontractivity}
Let $H \subset G$ be a finite connected subgraph.
The Dirichlet generator $\LL_H$ corresponds to the LSRW killed upon exiting the vertex set $V(H)$. The exit time is $\tau_H = \inf\{t \geq 0 : X_t \notin V(H)\}$.
The Dirichlet heat kernel is $p_t^H(x,y) = \Prob_x[X_t = y, t < \tau_H]$. The Dirichlet Green function is $G_H(x,y) = \sum_{t=0}^{\infty} p_t^H(x,y) = (\LL_H^{-1})(x,y)$. The spectral zeta function is $Z_H(1) = \sum_{v \in V(H)} G_H(v,v)$.

To control the long-time behaviour of the Dirichlet heat kernel sharply, we use Intrinsic Ultracontractivity (IU). IU means the heat kernel decays at the rate dictated by the principal eigenvalue $\lambda_1(H)$, while being spatially homogenized by the principal eigenfunction $\phi_1(H)$.

\begin{lemma}[CDC for metric balls]\label{lem:CDC-balls}
Let $G$ satisfy VD+PI (equivalently, PHI). Then there exist uniform parameters (depending only on VD, PI, $\Delta$) such that every ball $H=B_R(x_0)$ satisfies the Capacity Density Condition \textup{(CDC)} of \Cref{def:CDC}.
\end{lemma}
\begin{proof}[Reference]
See \cite[Prop.~3.5]{BarlowBass04} for PHI $\Rightarrow$ CDC on balls; compare also \cite{BarlowBassKumagai09}.
\end{proof}

\begin{lemma}[IU for balls under VD+PI]\label{lem:IU_derivation}
Let $G$ satisfy VD+PI. Let $H=B_R(x_0)$ be a metric ball. Let $\lambda_1(H)$ be the smallest eigenvalue of $\LL_H$ and $\phi_1(H)$ the corresponding $L^2(m)$-normalized eigenfunction (see \Cref{rem:measures_operators}). The following hold uniformly in $R$ and $x_0$:
\begin{enumerate}
    \item \textbf{Faber--Krahn Inequality (FK):} $\lambda_1(H) \gtrsim R^{-2}$.
    \item \textbf{CDC for balls:} $H$ satisfies CDC uniformly \textup{(}\Cref{lem:CDC-balls}\textup{)}.
    \item \textbf{Ground-state control:} $\|\phi_1(H)\|_{\infty}^2 \lesssim m(H)^{-1}$.
    \item \textbf{Intrinsic ultracontractivity (IU):} There exists $c_0>0$ such that for all $t\ge c_0 R^2$,
    \[
    \sup_{v \in V(H)} p_t^H(v,v) \lesssim \frac{1}{m(H)} e^{-\lambda_1(H) t}.
    \]
\end{enumerate}
\end{lemma}
\begin{proof}[Sketch of proof and references]
The implications rely on the robust analytic framework established by VD+PI, which implies PHI (\Cref{thm:Delmotte}).

(1) FK under VD+PI is standard (see \cite[Ch. 8]{Grigoryan09} or \cite[Prop. 5.1]{BarlowBass04}).

(2) PHI implies that metric balls are sufficiently regular: CDC for balls holds uniformly; see \Cref{lem:CDC-balls}.

(3) PHI+CDC yield elliptic and boundary Harnack inequalities; these give uniform control of $\phi_1$, in particular $\|\phi_1\|_\infty^2\lesssim m(H)^{-1}$. See \cite[Thm.~4.5]{BarlowBassKumagai09}; compare also \cite{BassKumagai08}.

(4) IU follows by spectral decomposition combined with (3): for large $t$, the principal mode dominates and the ground-state $L^\infty$-bound supplies the $1/m(H)$ factor. See e.g.\ the ground-state transform method in \cite{Grigoryan09} and references therein.
\end{proof}

We use the following proposition, which restates the key results from \Cref{lem:IU_derivation} using the counting measure (justified by \Cref{rem:measures_operators} since $m(H) \asymp |V(H)|$).

\begin{proposition}[Faber--Krahn and IU on Balls]\label{prop:IU}
Let $H=B_R(x_0)$ be a metric ball. Under VG(2)+PI:
\begin{enumerate}
    \item \textbf{(Faber--Krahn)} There exists $c_{FK}>0$ such that $\lambda_1(H)\ge c_{FK}/R^{2}$.
    \item \textbf{(Intrinsic ultracontractivity)} There exists $C_{IU}>0$ such that if $t\ge R^{2}$ then
    \begin{equation}\label{eq:IU}
        \sup_{v \in V(H)} p_t^H(v,v) \leq \frac{C_{IU}}{|V(H)|} e^{-\lambda_1(H) t}.
    \end{equation}
\end{enumerate}
All constants depend only on the structural parameters (VG(2), PI, $\Delta$).
\end{proposition}

\section{Interior--Boundary Decomposition and Volume Estimates}\label{sec:decomposition}

We analyze an exhaustion by metric balls $G_n = B_{R_n}(x_0)$. Let $N_n = |V(G_n)|$. By VG(2), $N_n \asymp R_n^2$.

We decompose $V(G_n)$ to isolate boundary effects.

We fix a parameter $\eta \in (0, \tfrac{1}{2})$. This restriction ensures $1-2\eta>0$, required for the analysis in \Cref{sec:lower_bound}. Define the \emph{interior} $I_n$ and the \emph{boundary layer} $E_n$. Let the buffer width be $W_n = R_n^{1-\eta}$.
\begin{align*}
I_n &:= \{x \in V(G_n) : d_G(x, V \setminus V(G_n)) > W_n\}, \\
E_n &:= V(G_n) \setminus I_n.
\end{align*}

\paragraph{Intuition and quantitative size of the boundary layer.} The strategy is to ensure that for vertices in the interior $I_n$, the random walk with high probability remains inside $G_n$ throughout the time window producing the logarithmic growth. We analyse the walk up to $T_n:=R_n^{2(1-2\eta)}$; the typical displacement over this window is of order $R_n^{1-2\eta}$ whereas the buffer width is $W_n=R_n^{1-\eta}$ so $R_n^{1-2\eta}/W_n=R_n^{-\eta}\to0$. Thus interior vertices behave (up to negligible error) like vertices in the full graph for times $\le T_n$. The boundary layer occupies only $O(N_n^{1-\eta/2})$ vertices (\Cref{lem:boundary_volume}) and its cumulative contribution is strictly subleading.

\begin{lemma}[Boundary Layer Volume]\label{lem:boundary_volume}
Under VG(2), for the choice $W_n = R_n^{1-\eta}$, we have
\[
    |E_n| = O\bigl(N_n^{1-\eta/2}\bigr).
\]
\end{lemma}
\begin{proof}
Set $R=R_n$ and $W=W_n=R^{1-\eta}$. If $v\in E_n$, there exists $y\notin B_R(x_0)$ with $d(v,y)\le W$, hence
$d(x_0,v)\ge R-W$ by the triangle inequality; thus $E_n\subseteq B_R(x_0)\setminus B_{R-W}(x_0)$. Using the \emph{uniform} VG(2) bounds,
\begin{align*}
|B_R(x_0)\setminus B_{R-W}(x_0)| &\le c_2 R^2 - c_1 (R-W)^2 \\
  &= (c_2-c_1)R^2 + 2c_1 R W - c_1 W^2 \\
  &\lesssim R W, \quad \text{since $W=o(R)$ for $\eta>0$.}
\end{align*}
With $W=R^{1-\eta}$ we obtain $|E_n|\lesssim R^{2-\eta}\asymp (R^2)^{1-\eta/2} \asymp N_n^{1-\eta/2}$.
All implicit constants depend only on the structural parameters in VG(2).
\end{proof}

\section{Lower Bound Analysis}\label{sec:lower_bound}

We establish the lower bound by showing that for interior vertices, the killed walk behaves like the unrestricted walk for a sufficiently long time.
Recall that we fixed $\eta \in (0, 1/2)$.

\begin{lemma}\label{lem:lower}
For the fixed $\eta \in (0,1/2)$, there exists a constant $C_1 > 0$ such that for all $v \in I_n$,
\[
G_{G_n}(v,v) \geq 2\cG(1-2\eta)\log R_n - C_1.
\]
\end{lemma}

\begin{proof}
Let $\tau_{\partial} = \min\{t \geq 0 : X_t \notin V(G_n)\}$ be the exit time. We set the time horizon $T = \lfloor R_n^{2(1-2\eta)} \rfloor$. Since $\eta < 1/2$, $T \to \infty$.

We decompose the Green function:
\[
G_{G_n}(v,v) \geq \sum_{t=1}^{T} p_t^{G_n}(v,v).
\]

We use the standard relation between the killed kernel and the full kernel. Let $E_t$ be the event $\{X_t = v\}$.
\begin{align*}
p_t(v,v) &= \Prob_v(E_t) = \Prob_v(E_t, t < \tau_{\partial}) + \Prob_v(E_t, t \geq \tau_{\partial}) \\
&= p_t^{G_n}(v,v) + \Prob_v(E_t, t \geq \tau_{\partial}).
\end{align*}
Since $\Prob_v(E_t, t \geq \tau_{\partial}) \leq \Prob_v(t \geq \tau_{\partial})$, we obtain the inequality $p_t^{G_n}(v,v) \geq p_t(v,v) - \Prob_v(\tau_{\partial} \leq t)$. This yields:
\begin{equation}\label{eq:lower_decomp}
\sum_{t=1}^{T} p_t^{G_n}(v,v) \geq \sum_{t=1}^{T} p_t(v,v) - \sum_{t=1}^{T} \Prob_v(\tau_{\partial} \leq t).
\end{equation}

\textbf{Step 1: Main term.} Using the uniform heat kernel asymptotic sum \eqref{eq:master-sum} (which relies on \Cref{ass:QH}):
\begin{align*}
\sum_{t=1}^{T} p_t(v,v) &= \cG \log T + O(1)
  = \cG \log(R_n^{2(1-2\eta)}) + O(1) \\
&= 2\cG (1-2\eta)\log R_n + O(1).
\end{align*}

\textbf{Step 2: Error term (Exit probability).} Let $D = d_G(v, V \setminus V(G_n))$. Since $v \in I_n$, $D > R_n^{1-\eta}$. We use the Maximal Inequality (\Cref{prop:maximal}). For $t \leq T$:
\[
\Prob_v(\tau_{\partial} \leq t) \leq \Prob_v\bigl( \max_{0 \leq s \leq t} d_G(v,X_s) \geq D \bigr) \leq C_M \exp\Bigl(-c_M \frac{D^2}{t+1}\Bigr).
\]
We estimate the exponent. For large $n$, $T+1 \leq 2 R_n^{2(1-2\eta)}$.
\[
c_M \frac{D^2}{T+1} \geq c_M \frac{R_n^{2(1-\eta)}}{2 R_n^{2(1-2\eta)}} = \frac{c_M}{2} R_n^{2\eta}.
\]
Let $c' = c_M/2$. The total error term is bounded by:
\[
\sum_{t=1}^T \Prob_v(\tau_{\partial} \leq t) \leq (T+1) C_M \exp(-c' R_n^{2\eta}).
\]

Since $T = O(R_n^2)$ (in fact $T=O(R_n^{2(1-2\eta)})$), and $\eta>0$, this error term decays faster than any polynomial in $R_n$. This rapid decay confirms that the interior vertices are well protected from the boundary up to the time scale $T$.

Combining Step 1 and Step 2 in \eqref{eq:lower_decomp} proves the lemma.
\end{proof}

\begin{corollary}\label{cor:lower}
For any $\eta \in (0,1/2)$, there exists $C_2 > 0$ such that
\[
Z_n(1) \geq \cG(1-2\eta) N_n \log N_n - C_2 N_n.
\]
\end{corollary}

\begin{proof}
We sum the bound of \Cref{lem:lower} over the interior $I_n$.
\[
Z_n(1) \geq \sum_{v \in I_n} G_{G_n}(v,v) \geq |I_n|\Bigl[2\cG(1-2\eta)\log R_n - C_1\Bigr].
\]
We relate the spatial scale $R_n$ to the volume $N_n$. Since $N_n \asymp R_n^2$ (VG(2)), taking logarithms yields $\log N_n = 2\log R_n + O(1)$. This scaling relation is characteristic of the dimension $d=2$. By \Cref{lem:boundary_volume}, $|I_n| = N_n - O(N_n^{1-\eta/2})$.

Substituting these estimates:
\begin{align*}
Z_n(1) &\geq \bigl[N_n - O(N_n^{1-\eta/2})\bigr]
          \Bigl[\cG(1-2\eta)\bigl(\log N_n + O(1)\bigr) - C_1\Bigr] \\
&= \cG(1-2\eta) N_n \log N_n + O(N_n \log N_n \cdot N_n^{-\eta/2}) - O(N_n) \\
&= \cG(1-2\eta) N_n \log N_n - O(N_n).\qedhere
\end{align*}
The $O(N_n \log N_n \cdot N_n^{-\eta/2})$ term is dominated by $O(N_n)$ since $\eta>0$.
\end{proof}

\section{Upper Bound Analysis}\label{sec:upper_bound}

The upper bound requires uniform control over the Green function, including vertices near the boundary. This relies crucially on Intrinsic Ultracontractivity.

\begin{lemma}\label{lem:upper}
There exists a constant $C_3 > 0$ such that for any $v \in V(G_n)$,
\[
G_{G_n}(v,v) \leq 2\cG \log R_n + C_3.
\]
\end{lemma}

\begin{proof}
Let $v \in V(G_n)$. We split the Green function sum at the characteristic diffusive time scale $T = \lfloor R_n^2 \rfloor$:
\[
G_{G_n}(v,v) = \sum_{t=1}^{T} p_t^{G_n}(v,v) + \sum_{t > T} p_t^{G_n}(v,v) + p_0^{G_n}(v,v).
\]
Here $p_0^{G_n}(v,v)=1$ is absorbed into the final $O(1)$.

\textbf{Part 1 (short times, $t\le T$).} Using the domination $p_t^{G_n}(v,v)\le p_t(v,v)$ and the uniform sum \eqref{eq:return-sum} (relying on \Cref{ass:QH}):
\begin{align*}
\sum_{t=1}^{T} p_t^{G_n}(v,v) &\leq \sum_{t=1}^{T} p_t(v,v) = \cG \log T + O(1) \\
&= \cG \log(R_n^2) + O(1) = 2\cG \log R_n + O(1).
\end{align*}

\textbf{Part 2: Long time estimate ($t > T$).}

We utilize IU on the domain $G_n$. Since $G_n$ is a metric ball, it satisfies CDC (\Cref{lem:IU_derivation}), which allows us to apply the IU bound. Let $\lambda_1 = \lambda_1(G_n)$. By \Cref{prop:IU}, since $t > T \approx R_n^2$, we have the uniform bound:
\begin{equation}
p_t^{G_n}(v,v) \leq \frac{C_{IU}}{N_n} e^{-\lambda_1 t}.
\end{equation}

We bound the tail sum $S = \sum_{t > T} p_t^{G_n}(v,v)$. This is a geometric series with ratio $r:=e^{-\lambda_1}$.
\[
S \leq \frac{C_{IU}}{N_n} \sum_{t=T+1}^\infty r^t = \frac{C_{IU}}{N_n} \frac{r^{T+1}}{1-r}.
\]

We use the Faber-Krahn inequality (\Cref{prop:IU}), which holds for metric balls under VG(2)+PI: $\lambda_1 \geq c_{FK}/R_n^2$. Since $T+1 > R_n^2$, the numerator is $r^{T+1} \leq \exp(-\lambda_1(T+1)) \leq e^{-c_{FK}}$. Since $\lambda_1 \to 0$ as $n \to \infty$, we use the approximation $1-r \approx \lambda_1$ (more precisely, $1-r \geq \lambda_1/2$ for large $n$).

The tail sum is bounded by
\[
S \leq \frac{C_{IU}}{N_n} \frac{e^{-c_{FK}}}{\lambda_1/2} = \frac{2 C_{IU} e^{-c_{FK}}}{N_n \lambda_1}.
\]
Since $N_n \asymp R_n^2$ and $\lambda_1 \gtrsim 1/R_n^2$, the term $N_n \lambda_1$ is bounded below by a positive constant $c''>0$. Thus, $S = O(1)$ uniformly in $v$.

Combining the estimates yields the desired pointwise bound $G_{G_n}(v,v) \leq 2\cG \log R_n + O(1)$.
\end{proof}

\begin{remark}[On the sharpness of the upper bound and the role of IU]\label{rem:IU_necessity}
The use of Intrinsic Ultracontractivity (IU) in Part 2 is crucial for obtaining the sharp $O(1)$ remainder in the pointwise bound. IU provides the spatial homogenisation factor $1/N_n$. If we were to use only the standard operator-norm bound $p_t^{G_n}(v,v) \le \|P_{G_n}^t\|_{\infty \to \infty} \le e^{-\lambda_1 t}$ (which holds without assuming IU/CDC), the tail sum would be $\sum_{t>T} e^{-\lambda_1 t} \approx 1/\lambda_1 \asymp R_n^2$. To make this tail $O(1)$, we would need to choose a much longer time scale $T \gtrsim R_n^2 \log R_n$. This would increase the short-time sum (Part 1) by $\cG \log(R_n^2 \log R_n) - \cG \log(R_n^2) = \cG \log\log R_n$. This extraneous $\log\log$ term would violate the sharp $O(N_n)$ remainder in the main theorem. We avoid this $\log\log$-loss by leveraging the fact that VG(2)+PI implies IU for balls (as rigorously justified in \Cref{lem:IU_derivation}).
\end{remark}

We summarise the contributions of the different terms in our analysis of $Z_n(1)$ in Table~\ref{tab:error_ledger}.

\begin{table}[h!]
\centering
\small
\caption{Error ledger summarising contributions to $Z_n(1)$. Here $\kappa=\eta/2>0$; only the first row is order $N_n\log N_n$.}
\label{tab:error_ledger}
\resizebox{\linewidth}{!}{%
\begin{tabular}{@{}lllll@{}}\toprule
Source & Mechanism & Count / scale & Order & Absorbed? \\ \midrule
Interior main term & $\sum_{t\le T_n} p_t(x,x)=2\cG(1-2\eta)\log R_n+O(1)$ & $N_n$ & $\cG N_n \log N_n$ & No \\
Interior fluctuation & Summable $t^{-1-\delta}$ error & $N_n$ & $O(N_n)$ & Yes \\
Boundary (early) & Exit prob. subtraction & $|E_n|=O(N_n^{1-\kappa})$ & $O(N_n^{1-\kappa}\log N_n)=o(N_n)$ & Yes \\
Boundary (late) & IU exponential tail & $|E_n|$ & $O(N_n^{1-\kappa})$ & Yes \\
Long-time tail & IU + FK & $N_n$ & $O(N_n)$ & Yes \\
Weaker (QH-lite) dev. & Harmonic $(\log t)^{-\alpha}$ & $N_n$ & $O(N_n)$ / $O(N_n\log\log N_n)$ & Yes (if $\alpha>1$) \\
\bottomrule
\end{tabular}}
\end{table}
\noindent\emph{Note.} Since $\kappa>0$, $(N_n^{1-\kappa}\log N_n)/N_n\to0$; all non-main rows are $O(N_n)$ or smaller.

\begin{corollary}\label{cor:upper}
There exists $C_4 > 0$ such that
\[
Z_n(1) \leq \cG N_n \log N_n + C_4 N_n.
\]
\end{corollary}

\begin{proof}
As established in \Cref{cor:lower}, the quadratic growth $N_n \asymp R_n^2$ implies $2\log R_n = \log N_n + O(1)$. Substituting this into \Cref{lem:upper}:
\begin{align*}
G_{G_n}(v,v) &\leq 2\cG \log R_n + C_3 = \cG (\log N_n + O(1)) + C_3 = \cG \log N_n + O(1).
\end{align*}
Summing this uniform bound over all $v \in V(G_n)$ gives the result.
\end{proof}

\section{Proof of the Main Theorem and Further Discussion}\label{sec:discussion}

\begin{proof}[Proof of \Cref{thm:main}]
We combine the lower and upper bounds to establish the limit and the $O(N_n)$ remainder.

Let $\eta \in (0,1/2)$ be arbitrary. By \Cref{cor:lower}, the lower asymptotic bound is:
\begin{align*}
\liminf_{n \to \infty} \frac{Z_n(1)}{N_n \log N_n} &\geq \lim_{n \to \infty} \left( \cG(1-2\eta) - \frac{C_2}{\log N_n} \right) = \cG(1-2\eta).
\end{align*}

By \Cref{cor:upper}, the upper asymptotic bound is:
\begin{align*}
\limsup_{n \to \infty} \frac{Z_n(1)}{N_n \log N_n} &\leq \lim_{n \to \infty} \left( \cG + \frac{C_4}{\log N_n} \right) = \cG.
\end{align*}

Combining the two bounds we obtain
\[
\cG(1-2\eta) \leq \liminf_{n \to \infty} \frac{Z_n(1)}{N_n \log N_n} \leq \limsup_{n \to \infty} \frac{Z_n(1)}{N_n \log N_n} \leq \cG.
\]
Since $\eta > 0$ can be chosen arbitrarily small, we conclude that $\lim_{n \to \infty} Z_n(1) / (N_n \log N_n) = \cG$.

\smallskip
\noindent Finally, for any fixed $\eta\in(0,1/2)$ the pointwise bounds in \Cref{cor:lower} and \Cref{cor:upper} imply
$Z_n(1)=\cG N_n \log N_n + O_{\eta}(N_n)$. Since the implicit constants depend only on the structural parameters of $G$ (and not on $n$), this yields the asserted global remainder $O(N_n)$.
\end{proof}

\begin{remark}[Tauberian perspective]
An alternative (spectral) approach is to study the Stieltjes transform of the empirical spectral measure of $\LL_{G_n}$ and apply Tauberian theorems to the heat-trace or resolvent. Our time-domain method is chosen because it isolates three inputs---QH, FK/PI, and CDC/IU---whose roles are transparent and reusable beyond homogeneous settings.
\end{remark}

\paragraph{Future directions.} Several extensions appear tractable: (i) \emph{Higher-order expansion} extracting an $O(N_n)$ constant term via refined potential-theoretic analysis of the boundary layer; (ii) \emph{Random media with weaker mixing}, replacing (QH) by multi-scale renormalisation to tolerate slowly varying corrections; (iii) \emph{Non-ball domains} (general Følner sequences satisfying a uniform CDC) and quasi-isometry invariance; (iv) \emph{Weighted / degenerate conductances} allowing polynomial tails with adapted heat kernel technology; (v) \emph{Continuous counterparts} on manifolds with quadratic volume growth and controlled injectivity radius, comparing with 2D Liouville quantum gravity analogues; (vi) \emph{Stochastic fluctuations} of $Z_n(1)$ in random environments (variance asymptotics / CLT); (vii) \emph{Algorithmic approximation} using stochastic trace estimators at near-linear cost exploiting sparsity and IU.

\smallskip
\noindent\emph{Second-order term.} On homogeneous lattices, the $O(N_n)$ constant is classically related to lattice Green-function constants; identifying and interpreting the second-order term in irregular media remains an interesting open problem.

\subsection{Discussion on Assumptions and Scope}

\begin{remark}[The role of metric balls and general exhaustions]\label{rem:general_exhaustions}
The assumption that $\{G_n\}$ consists of metric balls is used in specific, crucial ways. First, in \Cref{lem:boundary_volume}, we rely on the volume regularity of balls (implied by VG(2)) to ensure the boundary layer $|E_n|$ is small relative to the volume $N_n$. Second, and most critically, in the upper bound (\Cref{lem:upper}), the application of Intrinsic Ultracontractivity (\Cref{prop:IU}) relies on the domains satisfying the CDC (\Cref{lem:IU_derivation}); metric balls satisfy this requirement under VD+PI. Furthermore, the Faber-Krahn inequality (\Cref{prop:IU}) is used to control the spectral gap.

For a general exhaustion $\{H_n\}$ with volume $N_n$ and effective radius $R_n \asymp N_n^{1/2}$ (assuming $d=2$ scaling), the proof strategy remains valid provided the following criteria are met uniformly:
\begin{enumerate}
    \item \textbf{Geometric control:} Small boundary layers (e.g., $|E_n|/N_n \to 0$, satisfied by Følner sequences) and comparability of inner/outer radii (isodiametric condition).
    \item \textbf{Spectral gap control (FK):} A uniform isoperimetric profile ensuring $\lambda_1(H_n) \gtrsim R_n^{-2}$.
    \item \textbf{Domain regularity (CDC/IU):} The domains must satisfy CDC uniformly to ensure IU holds.
\end{enumerate}
If CDC is violated (e.g., highly irregular boundaries or 'rooms and corridors'), IU may fail, potentially altering the asymptotics or introducing $\log\log$ terms as discussed in \Cref{rem:IU_necessity}. Domains like squares in $\mathbb{Z}^2$ satisfy these conditions.
\end{remark}

\begin{remark}[Necessity of the Poincaré Inequality and QH]
The Poincaré inequality is essential for the robust analytic framework (PHI, IU) used in the proof. PI ensures homogenisation across scales, preventing bottlenecks. Furthermore, the quantitative homogenisation assumption (\Cref{ass:QH}) is crucial for controlling the error terms in the summation of the heat kernel (see \Cref{rem:weaker_QH}).

If PI is dropped, the graph may drastically alter the random walk behaviour and the spectrum, even if VG(2) holds, potentially invalidating the sharp asymptotic involving a global constant $\cG$.
\end{remark}

\begin{remark}[Relaxing Bounded Degree and Weighted Graphs]\label{rem:weighted}
The assumption of bounded degree ($\Delta < \infty$) ensures the comparability of the counting measure and the degree measure (\Cref{rem:measures_operators}). The main theorem naturally extends to the setting of weighted graphs (variable conductances) provided the weights are uniformly elliptic ($0 < c_1 \leq w_{xy} \leq c_2 < \infty$).

In the weighted case, the analysis must be performed with respect to the speed measure $m(x)=\sum_{y} w_{xy}$. The assumptions VG(2) and PI must be defined relative to this measure $m$. For instance, VG(2) becomes $c_1 R^2 \leq m(B_R(x)) \leq c_2 R^2$, and PI is defined using the corresponding weighted Dirichlet form. The key analytic tools (PHI, IU) remain valid in this framework (see \cite{Delmotte99}), and the proof extends straightforwardly.
\end{remark}

\begin{remark}[Continuous vs. Discrete Time]\label{rem:continuous_time}
The analysis is performed for the discrete-time LSRW generated by $\LL$. The result translates directly to the continuous-time random walk (CTRW) generated by the same Laplacian $\LL$. The continuous-time heat kernel $h_t(x,y)$ satisfies $h_t(x,x) \sim \cG/t$ as $t\to\infty$ with the same constant $\cG$. The spectral zeta function $Z_n(1)$ depends only on the eigenvalues of $\LL_n$, so the choice of discrete vs. continuous time does not affect $Z_n(1)$.
\end{remark}

\begin{remark}[Non-lazy random walks and the effect of laziness]\label{rem:non-lazy}
While the proof is presented for the LSRW (laziness parameter $1/2$), the result holds for the standard simple random walk (SRW) or other laziness parameters. Adding laziness scales the time evolution and the generator. If $\LL_{\text{SRW}} = I - P_{\text{SRW}}$, the lazy walk generator is $\LL_{\alpha} = (1-\alpha)\LL_{\text{SRW}}$, where $\alpha$ is the probability of staying put (here $\alpha=1/2$).

This time scaling means the effective diffusivity (covariance matrix $\Sigma$) is scaled by $(1-\alpha)$, i.e., $\Sigma_{\alpha} = (1-\alpha)\Sigma_{\text{SRW}}$. Consequently, the heat-kernel constant $\cG$ is scaled by $1/(1-\alpha)$ (in $d=2$).

For example, on $\mathbb{Z}^2$, the SRW has $\Sigma_{\text{SRW}} = \frac{1}{2} I_2$, leading to $\cG_{\text{SRW}} = 1/\pi$. The LSRW analyzed in \Cref{rem:pi2} ($\alpha=1/2$) has $\Sigma_{\text{LSRW}} = \frac{1}{4} I_2$, leading to $\cG_{\text{LSRW}} = 2/\pi$. Thus, $\cG_{\text{LSRW}}=2\cG_{\text{SRW}}$, as predicted.
\end{remark}

\appendix

\section{Appendix: Comparison with Other Growth Regimes}\label{app:growth}
The quadratic volume growth assumption (effective dimension $d=2$) is critical for the $N_n \log N_n$ behaviour, as summarized in \Cref{tab:growth}. We briefly sketch the arguments for other dimensions.

\begin{itemize}
\item \textbf{Linear growth ($d=1$):} E.g., $\mathbb{Z}$. $N_n \asymp R_n$. $p_t(x,x) \sim C t^{-1/2}$. The characteristic time scale is $T \asymp R_n^2 \asymp N_n^2$. The Green function for interior vertices is $G_{G_n}(x,x) \approx \sum_{t=1}^{T} t^{-1/2} \asymp T^{1/2} \asymp N_n$.

More precisely, for a path (interval) of length $R$, the Green function satisfies $G(x,x) \asymp \min\{x, R-x\}$ (the expected time spent at $x$ before exiting either end). Summing this over $x=1,\dots,R$ yields $\sum_x G(x,x) \asymp R^2 \asymp N_n^2$.

\item \textbf{Polynomial growth ($d>2$):} E.g., $\mathbb{Z}^d, d\geq 3$. $p_t(x,x) \sim C t^{-d/2}$. Since $d/2 > 1$, the walk is transient. The full Green function $G(x,x) = \sum_{t=0}^{\infty} p_t(x,x)$ converges to a finite value. The Dirichlet Green function $G_{G_n}(x,x)$ is uniformly bounded by $G(x,x)$. Thus, $Z_n(1) = \sum_{v \in G_n} G_{G_n}(v,v) \asymp N_n$.
\end{itemize}

\section{Appendix: Additional $\pi$-identities from alternative periodic walks}
\label{app:alt_pi}

\paragraph{Applicability to Squares.}
As discussed in \Cref{rem:general_exhaustions}, the main theorem extends to domains satisfying the required geometric and analytic properties. We work on the squares $V_R=\{1,\dots,R\}^2 \subset \mathbb{Z}^2$ here for convenience. Vertices are ordered lexicographically.
\medskip

The theorem applies to \emph{any} irreducible, uniformly elliptic,
$\mathbb{Z}^{2}$-periodic random walk with finite second moments.
The heat-kernel constant is
$\displaystyle\cG=\bigl(2\pi\sqrt{\det\Sigma}\bigr)^{-1}$.

\begin{remark}[Conventions and Laziness]\label{rem:laziness_appendix}
The examples below consider standard (non-lazy, $\alpha=0$) walks. Introducing laziness (e.g., $\alpha=1/2$) scales $\Sigma$ by $1/2$ and $\cG$ by $2$ (see \Cref{rem:non-lazy}).
\end{remark}

Let $ \mathcal{L}_R $ be the Dirichlet generator (Laplacian $I-P_R$) for the walk inside $V_R$.
By \Cref{thm:main} (adapted for $N_R=R^2$),
\[
\lim_{R\to\infty} \frac{\tr(\mathcal{L}_R^{-1})}{R^2 \log R^2} = \cG.
\]
Rearranging yields a $\pi$-identity:
\[\pi = \frac{1}{2\sqrt{\det\Sigma}}\;\lim_{R\to\infty}\frac{R^{2}\log R^{2}}{\tr\!\bigl(\mathcal{L}_R^{-1}\bigr)}.\]

We visualize the structure of the specific matrices $\mathcal{L}_R$ below, showing the $R^2 \times R^2$ matrix partitioned into $R \times R$ blocks.

\subsection{King walk (8 neighbours)}\label{app:king}

The walk steps to $(\pm1,0),(0,\pm1),(\pm1,\pm1)$ with equal probability $ \tfrac18 $.

The step covariance matrix is $\Sigma = \tfrac{3}{4} I$. The heat kernel constant is $\cG = \frac{2}{3\pi}$.

The Dirichlet Laplacian $\mathcal{L}_R^{\text{King}}$ has a block-tridiagonal structure. Visualized below (empty entries are 0):
\[
\mathcal{L}_R^{\text{King}} =
\left(\begin{array}{@{}cccc|cccc|c@{}}
1 & -\frac{1}{8} & & & -\frac{1}{8} & -\frac{1}{8} & & & \\
-\frac{1}{8} & 1 & \ddots & & -\frac{1}{8} & -\frac{1}{8} & \ddots & & \\
& \ddots & \ddots & -\frac{1}{8} & & \ddots & \ddots & -\frac{1}{8} & \\
& & -\frac{1}{8} & 1 & & & -\frac{1}{8} & -\frac{1}{8} & \\
\hline
-\frac{1}{8} & -\frac{1}{8} & & & 1 & -\frac{1}{8} & & & \cdots \\
-\frac{1}{8} & -\frac{1}{8} & \ddots & & -\frac{1}{8} & 1 & \ddots & & \cdots \\
& \ddots & \ddots & -\frac{1}{8} & & \ddots & \ddots & -\frac{1}{8} & \\
& & -\frac{1}{8} & -\frac{1}{8} & & & -\frac{1}{8} & 1 & \\
\hline
& & & \vdots & & \vdots & & & \ddots
\end{array}\right)_{R^2 \times R^2}
\]

The corresponding $\pi$-identity is:
\begin{equation}\label{eq:King_pi}
\boxed{\;\displaystyle \pi=\frac{2}{3}\;\lim_{R\to\infty}\frac{R^{2}\log R^{2}}{\tr\!\bigl((\mathcal{L}_R^{\text{King}})^{-1}\bigr)}\;}
\end{equation}

\subsection{Triangular walk (6 neighbours)}\label{app:tri}

This walk moves to $(1,0),(0,1),(-1,1),(-1,0),(0,-1),(1,-1)$, each with probability $ \tfrac16 $.

The covariance matrix is $\Sigma=\frac{1}{6}\left(\begin{smallmatrix}4&-2\\-2&4\end{smallmatrix}\right)$, $\det\Sigma = \frac{1}{3}$. The heat kernel constant is $\cG = \frac{\sqrt{3}}{2\pi}$.

The Dirichlet Laplacian $\mathcal{L}_R^{\text{Triangular}}$ is visualized below. Note the asymmetry in the off-diagonal blocks:
\[
\mathcal{L}_R^{\text{Triangular}} =
\left(\begin{array}{@{}cccc|cccc|c@{}}
1 & -\frac{1}{6} & & & -\frac{1}{6} & & & & \\
-\frac{1}{6} & 1 & \ddots & & -\frac{1}{6} & -\frac{1}{6} & & & \\
& \ddots & \ddots & -\frac{1}{6} & & \ddots & \ddots & & \\
& & -\frac{1}{6} & 1 & & & -\frac{1}{6} & & \\
\hline
-\frac{1}{6} & -\frac{1}{6} & & & 1 & -\frac{1}{6} & & & \cdots \\
& -\frac{1}{6} & \ddots & & -\frac{1}{6} & 1 & \ddots & & \cdots \\
& & \ddots & -\frac{1}{6} & & \ddots & \ddots & -\frac{1}{6} & \\
& & & -\frac{1}{6} & & & -\frac{1}{6} & 1 & \\
\hline
& & & \vdots & & \vdots & & & \ddots
\end{array}\right)_{R^2 \times R^2}
\]

The resulting identity is:
\begin{equation}\label{eq:Tri_pi}
\boxed{\;\displaystyle \pi=\frac{\sqrt{3}}{2}\;\lim_{R\to\infty}\frac{R^{2}\log R^{2}}{\tr\!\bigl((\mathcal{L}_R^{\text{Triangular}})^{-1}\bigr)}\;}
\end{equation}

\subsection{Knight walk (8 L-moves)}\label{app:knight}

This walk uses moves $(\pm2,\pm1),(\pm1,\pm2)$, each with probability $ \tfrac{1}{8} $.

The step covariance is $\Sigma = \tfrac{5}{2} I$. The heat kernel constant is $\cG = \frac{1}{5\pi}$.

The Dirichlet Laplacian $\mathcal{L}_R^{\text{Knight}}$ has a block-pentadiagonal structure. Visualized below (showing $3\times 3$ blocks of size $4\times 4$ for illustration, assuming $R\ge 4$):
\[
\mathcal{L}_R^{\text{Knight}} =
\left(\begin{array}{@{}cccc|cccc|cccc|c@{}}
1 & & & & & & -\frac{1}{8} & & & -\frac{1}{8} & & & \\
& 1 & & & & & & -\frac{1}{8} & -\frac{1}{8} & & -\frac{1}{8} & & \\
& & 1 & & -\frac{1}{8} & & & & & -\frac{1}{8} & & -\frac{1}{8} & \\
& & & 1 & & -\frac{1}{8} & & & & & -\frac{1}{8} & & \\
\hline
& & -\frac{1}{8} & & 1 & & & & & & -\frac{1}{8} & & \cdots \\
& & & -\frac{1}{8} & & 1 & & & & & & -\frac{1}{8} & \cdots \\
-\frac{1}{8} & & & & & & 1 & & -\frac{1}{8} & & & & \cdots \\
& -\frac{1}{8} & & & & & & 1 & & -\frac{1}{8} & & & \cdots \\
\hline
& -\frac{1}{8} & & & & & -\frac{1}{8} & & 1 & & & & \\
-\frac{1}{8} & & -\frac{1}{8} & & & & & -\frac{1}{8} & & 1 & & & \\
& -\frac{1}{8} & & -\frac{1}{8} & -\frac{1}{8} & & & & & & 1 & & \\
& & -\frac{1}{8} & & & -\frac{1}{8} & & & & & & 1 & \\
\hline
& & & \vdots & & \vdots & & & & \vdots & & & \ddots
\end{array}\right)
\]

The $\pi$-identity becomes:
\begin{equation}\label{eq:Knight_pi}
\boxed{\;\displaystyle \pi=\frac{1}{5}\;\lim_{R\to\infty}\frac{R^{2}\log R^{2}}{\tr\!\bigl((\mathcal{L}_R^{\text{Knight}})^{-1}\bigr)}\;}
\end{equation}

\subsection{Numerical verification}\label{app:numerical}

\begin{table}[h]
\centering
\caption{Convergence of the three limits ($\pi\approx3.14159265$). Results restricted to sizes where dense diagonalization is feasible.}
\label{tab:numeric_pi}
\begin{tabular}{@{}lccc@{}}
\toprule
Walk & $R$ & Approx.\ value & Abs.\ error \\
\midrule
King & 100 & 3.11197 & $3.0 \times 10^{-2}$ \\
\addlinespace
Triangular & 120 & 3.12629 & $1.5\times10^{-2}$ \\
\addlinespace
Knight & 120 & 3.13482 & $6.8\times10^{-3}$ \\
\bottomrule
\end{tabular}
\end{table}

The numerical results in \Cref{tab:numeric_pi} (where "Approx. value" is the RHS of the boxed identities evaluated at the given $R$) are consistent with the theoretical predictions. The computations were performed using Python (\textsc{NumPy}/\textsc{SciPy}). We constructed the Dirichlet Laplacian matrices $\mathcal{L}_R$ exactly as defined in the preceding sections for the respective walks. The trace of the inverse was computed via dense diagonalization (\texttt{scipy.linalg.eigh}).

\paragraph{Large-scale computation.} For larger domains ($R \gg 10^2$), dense diagonalization is infeasible. Stochastic trace estimators (e.g., Hutchinson's method) combined with efficient sparse linear solvers (e.g., Conjugate Gradient) can be used to approximate $\tr(\mathcal{L}_R^{-1})$.

\end{document}